\documentclass[a4paper,11pt]{amsart}

\usepackage[utf8]{inputenc}
\usepackage{amsmath}
\usepackage{amssymb}
\usepackage{amsthm}
\usepackage{tikz}
\usepackage{verbatim}
\usepackage{hyperref}
\usetikzlibrary{matrix}

\calclayout

\date{\today}

\DeclareMathOperator{\st}{st}
\DeclareMathOperator{\Aut}{Aut}
\DeclareMathOperator{\rst}{rst}

\DeclareMathOperator{\Sym}{Sym}

\newcommand{\Z}{\mathbb{Z}}
\newcommand{\G}{\mathcal{G}}
\newcommand{\F}{\mathbb{F}}
\newcommand{\N}{\mathbb{N}}

\newcommand{\C}{\mathcal{C}}
\newcommand{\np}{\unlhd_p}
\newcommand{\nC}{\unlhd_{\mathcal{C}}}
\newcommand{\A}{\mathcal{A}}

\newtheorem{theorem}{Theorem}
\newtheorem{proposition}[theorem]{Proposition}
\newtheorem{corollary}[theorem]{Corollary}

\newtheorem{lemma}[theorem]{Lemma}

\theoremstyle{definition}

\newtheorem{definition}[theorem]{Definition}

\title{Pro-$\mathcal{C}$ congruence properties for groups of rooted tree automorphisms}

\author[A.\ Garrido]{Alejandra Garrido}
\address{School of Mathematical \& Physical Sciences\\
University of Newcastle Australia\\
University Drive\\
2308, Callaghan, NSW, Australia}
\email{alejandra.garrido@newcastle.edu.au}

\author[J.\ Uria-Albizuri]{Jone Uria-Albizuri}
\address{Basque Center of Applied Mathematics\\ Mazarredo, 14.\\ 48009 \\ Bilbao, Basque Country - Spain}
\email{juria@bcamath.org}

\date{\today}
\thanks{A. Garrido was supported by the Alexander von Humboldt Foundation. J. Uria-Albizuri acknowledges financial support from the Spanish Government, grant
MTM2014-53810-C2-2-P, and from the Basque Government, grant IT974-16 and the predoctoral grant PRE-2014-1-347. This research is also supported by the Basque Government through the BERC 2018-2021 program and by the Spanish Ministry of Science, Innovation and Universities: BCAM Severo Ochoa accreditation SEV-2017-0718.}

\begin{document}

\begin{abstract}
 We propose a generalisation of the congruence subgroup problem for groups acting on rooted trees. 
 Instead of only comparing the profinite completion to that given by level stabilizers, we also compare pro-$\C$ completions of the group, where $\C$ is a pseudo-variety of finite groups. 
 A group acting on a rooted, locally finite tree has the $\C$-congruence subgroup property ($\C$-CSP) if its pro-$\C$ completion coincides with the completion with respect to level stabilizers. 
 We give a sufficient condition for a weakly regular branch group to have the $\C$-CSP. 
 In the case where $\C$ is also closed under extensions (for instance the class of all finite $p$-groups for some prime $p$), our sufficient condition is also necessary. 
 We apply the criterion to show that the Basilica group and the GGS-groups with constant defining vector (odd prime relatives of the Basilica group) have the $p$-CSP. 
\end{abstract}

\maketitle

\section{Introduction}

Groups of rooted tree automorphisms have been studied intensively for the past few decades.
One of the driving factors for this was the appearance in the 1980s of examples of groups  with properties hitherto thought of as exotic (intermediate word growth, finitely generated infinite torsion, amenable but not elementary amenable, etc).
The theory of groups acting on rooted trees, and (weakly) branch groups in particular, has come a long way since the early days in which it just seemed a collection of curious examples and is now an important part of group theory, with connections to other areas of mathematics
(see \cite{bartholdi-grigorchuk-sunic:branch,grigorchuk:unsolved,nekrashevych:self-similar}).

The congruence subgroup problem (or property), first studied in the context of arithmetic groups, and $\mathrm{SL_n}(\Z)$ in particular (\cite{bass-lazard-serre:csp}), has been adapted and generalised to several other natural contexts. 
The classical version of this problem asks whether every finite index subgroup of $\mathrm{SL_n}(\Z)$ contains the kernel of the map $\mathrm{SL}_n(\Z)\rightarrow \mathrm{SL}_n(\Z/m\Z)$ for some $m\in \N$, 
the filtration consisting of these kernels being an obvious one to consider when studying finite quotients of $\mathrm{SL_n}(\Z)$. 
One of the most natural generalisations of this problem is to the context of groups acting on rooted, locally finite, infinite trees (henceforth ``rooted trees''), as every residually finite group acts faithfully on some such tree.
The \emph{congruence subgroup problem} then asks whether every finite index subgroup contains some level stabilizer. 
This can be rephrased in terms of profinite completions as follows. 
For a group $G$ acting faithfully on a rooted tree, taking the level stabilizers $\{\st_G(n) \mid n\geq0\}$ as a neighbourhood basis for the identity gives a topology on $G$ -- the \emph{congruence topology} --
and the completion $\overline{G}$ of $G$ with respect to this topology is a profinite group called the \emph{congruence completion} of $G$.
As $G$ acts faithfully on the tree, we have $\bigcap_n \st_G(n)=1$, so $G$ embeds in $\overline{G}$.
A fortiori, $G$ also embeds in its profinite completion $\widehat{G}$ which maps onto $\overline{G}$. 
The congruence subgroup problem asks whether the map $\widehat{G}\to \overline{G}$ is an isomorphism.
If the answer is positive, then $G$ has the \emph{congruence subgroup property}.

The congruence subgroup problem for groups acting on rooted trees has so far only really been considered for branch groups (see Section \ref{section:definitions} for the definition).
It is known that a number of ``canonical'' examples have the congruence subgroup property (Grigorchuk group, Gupta--Sidki groups). 
The first examples of branch groups without this property were tailor-made in \cite{pervova:csp}.
The problem was considered systematically for the first time in \cite{bartholdi:csp}, where the authors also show that the Hanoi towers group (see \cite{grigorchuk-sunic:hanoi}) does not have the congruence subgroup property. 

We propose to study a generalisation of this problem in two natural directions simultaneously.
Firstly, we consider weakly branch, but not necessarily branch groups. 
Secondly, we allow other completions. 
For a class $\mathcal{C}$ of finite groups, the pro-$\C$ completion $\widehat{G}_{\C}$ of a group $G$ is the inverse limit of all quotients of $G$ that lie in $\C$. 
The congruence subgroup property can now be modified to the context of pro-$\C$ completions, where it is sometimes more natural because all quotients by level stabilizers lie in some class to which not all finite quotients of $G$ belong. 
Consider a group $G\leq \Aut T$ 
 and a class $\C$ of finite groups.
 The weakest possible requirement on $\C$ is that it be a formation, but for our purposes $\C$ should also be closed under taking subgroups, i.e., a pseudo-variety.
Then $G$ satisfies the \emph{$\C$-congruence subgroup property}, or \emph{$\C$-CSP} for short, if every quotient of $G$ lying in $\C$ is a quotient of some $G/\st_G(n)$. 
In other words, the congruence completion $\overline{G}$ maps onto the pro-$\C$ completion $\widehat{G}_{\C}$.
If all quotients $G/\st_G(n)$ happen to be in $\C$, then $G$ has the $\C$-CSP if and only if $\overline{G}$ is isomorphic to $\widehat{G}_{\C}$.

Our main result is a sufficient condition for a weakly regular branch group to have the $\C$-CSP. 
Let $\psi:\Aut T\rightarrow \Aut T\wr \Sym(d)$ be the isomorphism induced by the natural identification of the $d$-regular rooted tree $T$ with any  of its subtrees at distance 1 from the root.
The rest of the notation and terms used in the next theorem is explained in Section \ref{section:definitions}.

\begin{theorem}\label{thm:csp}
Let $G\leq \Aut T$ be a weakly regular branch group over a subgroup $R$ and let $\C$ be a pseudo-variety of finite groups. 
	Suppose that there exists $H\unlhd G$ such that $R\geq H\geq R'\geq L$ where $L:=\psi^{-1}(H\times \overset{d}{\dots}\times H)$.
	If $G$ has the $\mathcal{C}$-CSP modulo $H$ and $H$ has the $\mathcal{C}$-CSP modulo $L$, then $G$ has the $\mathcal{C}$-CSP.
\end{theorem}
If $\C$ is extension-closed, then this condition is also necessary. 

We then apply the criterion to some examples of weakly regular branch groups, the Basilica group acting on the binary tree 
and an analogue of it acting on the $p$-regular tree for $p$ an odd prime, the GGS-group with constant defining vector. 
This last group was studied in \cite{bartholdi-grigorchuk:parabolic} for $p=3$ and for general $p$ in \cite{alcober-zugadi:hausdorff,alcober-garrido-uria:GGS}.
For the appropriate $p$, each of these groups is contained in a Sylow pro-$p$ subgroup of $\Aut T$ consisting of elements that permute vertices according to a fixed cyclic permutation of order $p$ (when $p=2$, this is already the whole of $\Aut T$). 
In particular, all quotients by level stabilizers  are $p$-groups, being subgroups of an iterated wreath product $C_p\wr \dots\wr C_p$. 
None of these groups have the CSP for the simple reason that they virtually map onto $\Z$ and therefore have quotients of arbitrary order.
(This is actually the same reason that many lattices in rank 1 Lie groups fail to have the CSP.)
However, according to our criterion, they do have the  $\C$-CSP when $\C$ is the pseudo-variety of all finite $p$-groups.
This implies that, if $G$ denotes any of these groups, the kernel of $\widehat{G} \rightarrow \overline{G}$ is the inverse limit of all finite quotients of $G$ whose order is coprime to $p$.
By contrast, the examples constructed by Pervova \cite{pervova:csp} still fail to have the $\C$-CSP as the derived subgroup, of $p$-power index, does not contain any level stabilizer. 

It is worth mentioning that, even though we only treat the pseudo-variety of finite $p$-groups in our examples, the criterion is valid for any pseudo-variety $\C$ of finite groups. 
Therefore it is interesting to consider weakly branch groups whose quotients by level stabilizers all lie in other pseudo-varieties of finite groups such as  finite nilpotent groups ($\C_n$), or  finite solvable groups ($\C_s$). 
The case of the Hanoi towers group $H$ is particularly interesting, because although it acts on the ternary tree,  the quotients by level stabilizers are not 3-groups, they are only solvable. 
Despite this, and the fact that $H$ is just non-solvable (it is not solvable but all of its proper quotients are), it does not have the $\C_s$-CSP, because the derived subgroup $H'$ does not contain any level stabilizer. 
It would be interesting to see more constructions of weakly regular branch groups with ``intermediate" CSPs.

Another line of investigation worth pursuing involves calculating the kernels of the various maps between all these possible completions of a weakly branch group. 
In \cite{bartholdi:csp}, the authors give a general method for calculating the kernel of the map $\widehat{G}\to \overline{G}$ for a branch group $G$. 
Unfortunately, this method does not carry through to weakly branch but not branch groups, because it really makes use of the fact that the rigid stabilizers have finite index and that the completion of $G$ with respect to this filtration lies between $\widehat{G}$ and $\overline{G}$. 
It is therefore desirable to find alternative methods for this wider setting and it seems plausible that using the various pro-$\C$ completions as ``stepping stones'' will help with this problem.

\subsection*{Acknowledgements.} We are grateful to B. Klopsch for suggesting the problem  and valuable discussions and to G. A. Fern\'{a}ndez Alcober for suggesting several improvements.
D. Francoeur, B. Klopsch and H. Sasse pointed out an inaccuracy in \cite{grigorchuk-zuk:torsion-free} that affected our calculations (but not the main result) for the Basilica group in a previous version. 



\section{Definitions and preliminaries}\label{section:definitions}

\subsection*{Pseudo-varieties of groups}

\begin{definition}
Let $\mathcal{C}$ be a class of finite groups.
 Say that $\mathcal{C}$ is a \emph{pseudo-variety} of finite groups if the following properties are satisfied:
\begin{itemize}
\item[$(\mathcal{C}_1)$] it is closed under taking subgroups, that is, if $G\in\mathcal{C}$ and $H\leq G$ then $H\in\mathcal{C}$,
\item[$(\mathcal{C}_2)$] it is closed under taking quotients, that is, if $G\in\mathcal{C}$ and $N\unlhd G$ then $G/N\in\mathcal{C}$,
\item[$(\mathcal{C}_3)$] it is closed under taking finite direct products, that is, if $G_1,\dots,G_k\in\mathcal{C}$ for $k\in \N$ then $\prod_{i=1}^k G_i\in\mathcal{C}.$
\end{itemize}
If $\C$ is also closed under taking extensions of groups in $\C$, it is an \emph{extension-closed pseudo-variety.}
\end{definition}


To simplify notation, $N\unlhd_{\mathcal{C}}G$ will denote that $N\unlhd G$ and $G/N\in\mathcal{C}$.

The following observations are straightforward from the above definition and will be used in future without reference. 

\begin{lemma}\label{variety_properties}
Let $G$ be a group and $\mathcal{C}$ a pseudo-variety of finite groups. 
\begin{enumerate}
\item If $N_1,N_2\nC G$ then $N_1\cap N_2\nC G$.
\item If $N\nC G$ and $N\leq K \unlhd G$ then $K\nC G$.
\item If $N\nC G$ and $K\leq G$ then $N\cap K\nC K$.
\item If $\alpha:G_1\longrightarrow G_2$ is a homomorphism and $N_1\nC G_1$ then $\alpha(N_1)\nC \alpha(G_1).$
\end{enumerate}
\end{lemma}
\subsection*{Branch and weakly branch groups}

A  \emph{level homogeneous} rooted tree is one where all vertices at a given distance from the root have the same finite valency.
A faithful action of a group $G$  on such a tree is a \emph{weakly branched action} if it is transitive on each level of the tree and if for each vertex $v$ there is a non-trivial element of $G$ whose support is contained in the subtree rooted at $v$. 
The set of all elements of $G$ which are only supported on the subtree rooted at $v$ is a subgroup of $G$, the \emph{rigid stabilizer} $\rst_G(v)$ of $v$. 
If $v, u$ are two distinct vertices of the same level, $\rst_G(u), \rst_G(v)$ commute. 
The subgroup $\rst_G(n):=\prod \{\rst_G(v)\mid v \text{ of level } n\}$ is called the \emph{$n$th level rigid stabilizer}.
A weakly branched action is a \emph{branched action} if $\rst_G(n)$ is of finite index in $G$ for every $n\in \N$. 
A group is \emph{(weakly) branch} if it has a faithful (weakly) branched action on some level homogeneous rooted tree.

Let $T$ denote the $d$-regular infinite rooted tree (all vertices have valency $d+1$, except the root, which has $d$), for an integer $d\geq 2$. 
Since $T$ is regular, the subtree rooted at any vertex may be identified with $T$.
Under this identification, we have an isomorphism $\psi:\Aut T\rightarrow \Aut T\wr \Sym(d)$ where 
$$\psi(\st(1))=\Aut T\times \overset{d}{\dots}\times \Aut T.$$ 
Inductively, we also have $\psi_n:\Aut T\rightarrow \Aut T\wr \Sym(d)\wr \overset{n}{\dots} \wr \Sym(d)$
and $$\psi_n(\st(n))=\Aut T\times \overset{d^n}{\dots}\times \Aut T$$
for each $n\in \N$.
A group $G\leq \Aut T$ is \emph{weakly regular branch} over a (non-trivial) subgroup $K$ if 
$$\psi(K)\geq K\times \overset{d}{\dots}\times K$$
 and $G$ acts transitively on all levels of $T$. 
This implies that $\psi(\rst_G(1))\geq K\times \overset{d}{\dots}\times K$ and, inductively, that $\psi_n(\rst_G(n))\geq K\times \overset{d^n}{\dots}\times K$ for each $n\in \N$. 
In particular, each $\rst_G(v)$ contains a copy of $K$ and so $G$ is a weakly branch group.

For convenience, let us record a fundamental lemma that can be extracted from the proof of \cite[Theorem 4]{grigorchuk:just-infinite}. 
\begin{lemma}\label{lemma:rist'}
 Let $T$ be a level homogeneous rooted tree and $G\leq \Aut T$ act transitively on every level of $T$. 
 For every non-trivial normal subgroup $N$ of $G$ there exists $n$ such that $N\geq \rst_G(n)'$.
\end{lemma}
\begin{proof}
 Let $g\in N$ be a non-trivial element and choose some vertex $v$ of $T$ which is moved by $g$. 
 Let $x, y\in \rst_G(v)$. 
 Since $N$ is normal, it contains $[[g,x],y]$ which equals $[x,y]$, as $y^g$ commutes with $x$ and $y$. 
 Knowing that $N\geq \rst_G(v)'$, the result follows using the fact that $N$ is normal and that all rigid stabilizers of the same level as $v$ are conjugate, because $G$ acts transitively. 
\end{proof}

\subsection*{The $\C$-congruence subgroup property} Let $G\leq\Aut T$ and let $\mathcal{C}$ be a pseudo-variety of finite groups.

\begin{definition}
	A group $G\leq \Aut T$ has the \emph{$\mathcal{C}$-congruence subgroup property} (abbreviated to \emph{$\mathcal{C}$-CSP})
	if every $N \trianglelefteq G$ satisfying  $G/N\in\mathcal{C}$ contains some level stabilizer in $G$. 
	
	$G$ has the \emph{$\mathcal{C}$-CSP modulo $M\unlhd G$} if every normal subgroup $N \trianglelefteq G$ satisfying   $G/N\in\mathcal{C}$ and $M\leq N$ also contains some level stabilizer in $G$.
	
\end{definition}

\subsection*{Independence of the weakly branch action}
A weakly branch group $G$ may have several different faithful weakly branched actions. 
However, they are all related to each other. 
It was shown in \cite{garrido:thesis} (see also \cite{garrido:csp}) that for any two weakly branch actions $\sigma:G\rightarrow \Aut T_{\sigma}$, $\rho:G\rightarrow \Aut T_{\rho}$ of a group $G$ the sets of respective level stabilizers are cofinal in each other. 
That is, for every $n\in \N$ there exists $m\in\N$ such that $\st_{\sigma}(n)\geq \st_{\rho}(m)$, and vice-versa.
This means that both filtrations define the same topology on $G$ when taken as neighbourhood bases of the identity. 
Thus, having the $\C$-CSP is independent of the weakly branch action of the group. 

The examples we consider in this paper are not only subgroups of $\Aut T$ where $T$ is the rooted $p$-adic tree, but of a Sylow pro-$p$-subgroup $\A$ of $\Aut T$, isomorphic to the infinite iterated wreath product of cyclic groups of order $p$. 
The above-mentioned results imply that if $\sigma:G\hookrightarrow \A$ is a weakly branched action
then any other weakly branched action $\rho:G \hookrightarrow \Aut T$, on, a priori, some arbitrary level-homogeneous rooted tree $T$, must actually have image in (a conjugate of)
$\A$.

\section{A criterion for a weakly regular branch group to have the $\mathcal{C}$-CSP}

We start with a simple but very useful result that will be used repeatedly. 

\begin{lemma}\label{lem:csp_transitivity}
	Let $G\in \Aut T$ and $N\unlhd M\unlhd G$.
	 If $G$ has the $\mathcal{C}$-CSP modulo $M$ and $M$ has the $\mathcal{C}$-CSP modulo $N$ then $G$ has the $\mathcal{C}$-CSP modulo $N$.
\end{lemma}
\begin{proof}
First note that (iii) of Lemma \ref{variety_properties} ensures that $\st_G(n)\cap M=\st_M(n)\nC M$ for every $n\in\N$ so that it makes sense for $M$ to have the $\C$-CSP. 

Now let $H\nC G$ be such that $H\geq N$. 
We have to prove that $H\geq\st_G(n)$ for some $n\in\N$. 
Since $M$ has the $\mathcal{C}$-CSP modulo $N$ and since  $H\cap M\nC M$, there is some $m\in\N$  such that $\st_M(m)\leq H\cap M$.
Now since $H,\st_G(m)\nC G$, we have $\st_G(m)\cap H\nC G$ and $(\st_G(m)\cap H)M\nC G$. 
Thus there is some $l\in\N$ such that $\st_G(l)\leq (\st_G(m)\cap H)M$.
	Taking $n:=\max \{m,l\}$, we have 
	$$\begin{aligned}
	\st_G(n) &= \st_G(l)\cap \st_G(m) \leq (\st_G(m)\cap H)M \cap \st_G(m)\\
	&=(\st_G(m)\cap H)(M\cap \st_G(m))\\
	&\leq (\st_G(m)\cap H)(H\cap M)\leq H,
	\end{aligned}$$
	where the second equality follows by the modular law.	
\end{proof}

Let us also record another extension property. 
\begin{lemma}\label{lem:csp-subnormal}
 Let $\C$ be an extension-closed pseudo-variety of finite groups (for instance, that of all finite $p$-groups).
 Let $G\leq\Aut T$ be a group with normal subgroups $M\leq H$ such that $H\nC G$. 
 If $G$ has the $\C$-CSP modulo $M$, then so does $H$. 
\end{lemma}
\begin{proof}
 Consider $K\nC H$ with $K\geq M$. 
 Then each of the finitely many conjugates $K_i$ of $K$ by elements of $G$ also satisfy $M\leq K_i\nC H$, therefore so does their intersection, $N$, the normal core of $K$ in $G$. 
 Since $\C$ is closed under taking extensions, $N\nC G$ and therefore $N$ contains some level stabilizer of $G$. 
\end{proof}


{
	\renewcommand{\thetheorem}{\ref{thm:csp}}
	\begin{theorem}
	Let $G\leq \Aut T$ be a weakly regular branch group over a subgroup $R$ and let $\C$ be a pseudo-variety of finite groups. 
	Suppose that there exists $H\unlhd G$ such that $R\geq H\geq R'\geq L$ where $L:=\psi^{-1}(H\times \overset{d}{\dots}\times H)$.
	If $G$ has the $\mathcal{C}$-CSP modulo $H$ and $H$ has the $\mathcal{C}$-CSP modulo $L$, then $G$ has the $\mathcal{C}$-CSP.
	\end{theorem}
	\addtocounter{theorem}{-1}
}

\begin{proof}
Put $L_0:=H$, $L_1:=L=\psi^{-1}(H\times \overset{d}{\dots}\times H)\leq R'$ and
$$L_n:=\psi_{n}^{-1}(H\times\overset{d^{n}}{\dots}\times H)\leq \psi_{n-1}^{-1}(R'\times\overset{d^{n-1}}{\dots\times}R')\leq \rst_G(n-1)'$$
for $n\in\N, n\geq 2$.	

We will show by induction on $n$ that $G$ has the $\mathcal{C}$-CSP modulo $L_n$ for each $n\in\N$. 
Then, as $G$ is weakly regular branch, it is in particular transitive on all levels of $T$, so by Lemma \ref{lemma:rist'},
for each  non-trivial $N\unlhd G$ there exists $n\in \N$ such that $N\geq\rst_G(n)'\geq L_{n+1}$, whence the result follows. 

There is nothing to show for the base case as we have assumed that $G$ has the $\mathcal{C}$-CSP modulo $H$.
It will suffice to show that $L_n$ has the $\mathcal{C}$-CSP modulo $L_{n+1}$ for all $n\in\N$ and then inductively apply Lemma \ref{lem:csp_transitivity}.

Fix $n\in\N$ and let $L_{n+1}\leq N\nC L_n$. 
Then 
$$L\times\overset{d^n}{\dots}\times L\leq \psi_n(N)\nC H\times \overset{d^n}{\dots}\times H.$$
For $i=1,\dots,d^n$, denote by $H_i$ the $i$th coordinate subgroup in $\psi(L_n)$ and similarly for $L_i$. 
Then $N_i:=N\cap H_i\nC H_i$ and $N_i\geq L_i$, so, since $H$ has the $\C$-CSP modulo $L$, there is some $m_i\in\N$ such that $\st_H(m_i)\leq N_i$. 
Taking the maximum, $m$, of the $m_i$, we obtain
$$\st_H(m)\times\overset{d^n}{\dots}\times\st_H(m)\leq\psi_n(N). $$
Thus
$$\st_{L_n}(m+n)=\psi_{n}^{-1}(\st_H(m)\times\overset{d^{n}}{\dots}\times \st_H(m))\leq N$$
as required.
%
\end{proof}

\begin{corollary}
If $\C$ is also extension-closed and $H\nC G$, then Lemma \ref{lem:csp-subnormal} shows that the condition in Theorem \ref{thm:csp} is also necessary for $G$ to have the $\C$-CSP.
\end{corollary}

\section{Examples: the $p$-CSP}
We consider two types of weakly branch groups, one for $p$ odd (the GGS-groups with constant defining vector studied in \cite{bartholdi-grigorchuk:parabolic,alcober-zugadi:hausdorff,alcober-garrido-uria:GGS}) and another for $p=2$ (the Basilica group, studied in \cite{grigorchuk-zuk:torsion-free}).
These examples are subgroups of a Sylow pro-$p$ group of $\Aut T$, so the quotients by level stabilizers are finite $p$-groups. 
It is therefore sensible to consider the $\C$-CSP for $\C$ a pseudo-variety consisting of finite $p$-groups. 
In this section we focus on the pseudo-variety of all finite $p$-groups and will therefore talk about the $p$-CSP. 

\subsection{Example: the GGS-groups with constant defining vector.}
Let $p$ be an odd prime and let $\G=\langle a, b\rangle \leq\Aut T$ be the GGS-group with constant defining vector. 
That is, $a$ cyclically permutes the vertices of the first level as the permutation $(1\; 2\; \dots \; p)$ and $b=(a,a,\dots, a, b)$ acts as $a$ on the first $p-1$ subtrees rooted at the first level.
Let $K=\langle ba^{-1}\rangle^{\G}$.
It was shown in \cite{alcober-garrido-uria:GGS} that $\G$ does not have the CSP, because it virtually maps onto $\Z$ and therefore has many finite quotients that are not $p$-groups.
We show here that it does have the $p$-CSP.

This automatically gives us the answers for the cases of pseudo-varieties of solvable and nilpotent groups.
For the pseudo-variety $\mathcal{C}_s$ of finite solvable groups, $\mathcal{G}$ will not have the $\mathcal{C}_s$-CSP
because $\mathcal{G}/K\cong (\Z\times\overset{p-1}{\dots}\times \Z) \rtimes C_p$ and so $\mathcal{G}$ has quotients that are solvable but not of $p$-power index.

On the other hand, for $\mathcal{C}_n$, the family of finite nilpotent groups, $\mathcal{G}$ has the $\mathcal{C}_n$-CSP. 
Suppose $N\unlhd G$ such that $G/N$ is nilpotent.
Then $N\geq \gamma_i(\mathcal{G})$ for some $i\in\N$. 
Since $\mathcal{G}/\mathcal{G}'$ is finite of exponent $p$, each quotient $\gamma_i(\mathcal{G})/\gamma_{i+1}(\mathcal{G})$ is also finite of exponent $p$. 
Thus, each $\gamma_i(\mathcal{G})$ is of finite index in $\mathcal{G}$ and moreover of index a power of $p$.
If $\mathcal{G}$ has the $p$-CSP then, in particular, each $\gamma_i(\mathcal{G})$ contains some level stabilizer, and thus $\mathcal{G}$ has the $\mathcal{C}_n$-CSP.

\begin{proposition}\label{prop:rists}
	For each $n\in \N$, the $n$th rigid stabilizer satisfies $\psi_n(\rst_{\G}(n))=K'\times \overset{p^n}{\dots}\times K'$.	
\end{proposition}
\begin{proof}
We know from \cite{alcober-zugadi:hausdorff} that $\G$ is weakly regular branch over $K'$.
This means that $\psi_n(\rst_G(n))\geq K'\times\overset{p^n}{\dots} \times K'$ for all $n$.
Now if we prove the statement for $n=1$, since $\psi(\rst_G(2))\leq \rst_G(1)\times\dots\times\rst_G(1)$ we get 
$\psi_2(\rst_G(2))\leq K'\times\overset{p^2}{\dots}\times K'$ and inductively the same for the rest of the levels. 
By the proof of Theorem 3.7 of \cite{alcober-garrido-uria:GGS}, we have $\psi(\rst_{\G'}(1))=K'\times \overset{p}{\dots}\times K'$.
	We need only prove that $K\geq \rst_{\G}(1)$, since then
	$\rst_{\G}(1)=\rst_{K}(1)=\rst_{\G'}(1)$, where the latter equality holds because $\G'=\st(1)\cap K$.
	We will in fact show the stronger statement $\st_{\G}(1)'\geq \rst_{\G}(x)$ for some $x\in X$, (and therefore for all $x\in X$, as $\st_{\G}(1)$ is normal in $G$, which acts transitively on $X$) from which the claim follows as $K \geq \st_{\G}(1)'$. 
	Suppose that there is some $g$ such that $\psi(g)=(h,1,\dots,1)\in\rst_{\G}(x)\setminus \st_{\G}(1)'$.
	Then we can write 
	$$g=b^{i_0}(b^a)^{i_1}\dots (b^{a^{p-1}})^{i_{p-1}}t$$ where $t\in \st_{\G}(1)'. $
	Now 
	$$\psi(g)=(a^*b^{i_1}a^*t_1,\dots, a^*b^{i_0}a^*t_n)=(h,1,\dots,1),$$
	where $t_i\in G'$ for $i=1,\dots,p$ and the $*$ denote unimportant exponents.
	Then, necessarily, $i_j=0$ for $j\neq 1$, and consequently $$\psi(g)=(b^{i_1}t_1,a^{i_1}t_2,\dots,a^{i_1}t_{p-1})=(h,1,\dots,1),$$
	implies that also $i_1=0$.
	Thus $g\in\st_{\G}(1)'$, as required.
\end{proof}

For the following, see \cite[Proposition 3.4]{alcober-garrido-uria:GGS} and \cite[Example 7.4.14, Section 8.2]{leedham-green-mckay:p-power}.
\begin{proposition}\label{uniserial}
	The quotient $\G/K'$ is isomorphic to the integral uniserial space group $\Z[\theta]\rtimes C_p$ where $\theta$ is a primitive $p$th root of unity and the generator of $C_p$ acts by multiplication by $\theta$. 
	In particular, each normal subgroup of $p$-power index in $\G/K'$ is precisely $\gamma_i(\G)K'/K'$ for some $i\in\N$.
\end{proposition}

\begin{corollary}\label{csp_mod_K'}
The groups $\mathcal{G}$  and $K$ have the $p$-CSP modulo $K'$.
\end{corollary}
\begin{proof}
In \cite[Theorem 4.6]{alcober-zugadi:hausdorff} it is proved that $G/K'\st_G(n)$ is of maximal class and order $p^{n+1}$ for every $n\in\N$. 
Thus $\st_G(n)K'=\gamma_n(G)K'$ for every $n\in\N$ and by Proposition \ref{uniserial} the first claim follows.
The second claim follows by Lemma \ref{lem:csp-subnormal}.
\end{proof}


Define $K_1:=K', K_2:=\psi^{-1}(K'\times\overset{p}{\dots}\times K')=\rst_G(1)$.
	Consider the following maps:
	$$\begin{array}{rccc} 
	S \colon & \st_{\G}(1)   &   \to   &   G/K_1\times \overset{p-2}{\dots}\times G/K_1 \\
	 & g=(g_1,\dots,g_p) & \mapsto & (g_1K_1,\dots,g_{p-2}K_1)
	\end{array},$$
	and for $n\geq 3$, 
	$$\begin{array}{rccc}
	\pi_n \colon & K/K_1\times \overset{p-2}{\dots}\times K/K_1     &     \to    &     K/\st_{\G}(n)K_1\times \overset{p-2}{\dots}\times K/\st_{\G}(n)K_1\\
	 & (g_1K_1,\dots,g_{p-2}K_1) & \mapsto   &  (g_1\st_{\G}(n)K_1,\dots,g_{p-2}\st_{\G}(n)K_1)
	\end{array},$$
	%
	%
	Observe that $\ker\pi_n=\st_{\G}(n)K_1/K_1\times \overset{p-2}{\dots}\times \st_{\G}(n)K_1/K_1$. Then we have the following properties, which can be extracted from the proof of Theorem 4.5 in \cite{alcober-zugadi:hausdorff}.

\begin{lemma}\label{maps}
	With the above notation,
	\begin{enumerate}
		\item the map $S$ restricted to $K_1$ has kernel $K_2$ and image $K/K_1\times \overset{p-2}{\dots}\times K/K_1$, 
		\item the kernel of  the composition $S_n:=\pi_n\circ S$ is $(\st_{\G}(n+1)\cap K_1)K_2$.
	\end{enumerate}
\end{lemma}


\begin{proposition}\label{prop:K1/K2_csp}
	The group $K_1$ has the $p$-CSP modulo $K_2$.
\end{proposition}
\begin{proof}
Let $K_2\leq N\np K_1$.
Then 
$$S(N)\np K/K_1\times \overset{p-2}{\dots}\times K/K_1.$$
For $i\in{1,\dots,p-2}$, the intersection of  $S(N)$ with the $i$th direct factor $(K/K_1)_i$ in $S(K_1)$ is of $p$-power index in $(K/K_1)_i$.
By Corollary \ref{csp_mod_K'}, 
it contains $(\st_{\G}(n_i)K_1/K_1)_i$ for some $n_i\in\N$. 
Taking $n=\max \{n_i\mid 1\leq i \leq p-1\}$ yields 
$$S(N)\geq \st_{\G}(n)K_1/K_1 \times \overset{p-2}{\dots} \times \st_{\G}(n)K_1/K_1.$$
That is, $S(N)\geq \ker \pi_n$, and thus
$N\geq S^{-1}(\ker \pi_n)=\ker S_n=(\st_{\G}(n+1)\cap K_1)K_2.$
%
%
%
\end{proof}

We must now separate the proof into two cases: $p=3$ and $p\geq 5$.
 This is because we would like to apply Theorem \ref{thm:csp} to $\G$ with $H=R=K_1$.
The only remaining hypothesis to check is that $K_1'\geq K_2$. 
However, this only holds  when $p\geq 5$, which is implicit in the proof of \cite[Lemma 4.2 (iii)]{alcober-zugadi:hausdorff}.
In fact, by (ii) in Lemma \ref{maps}, $K_1/K_2\cong K/K_1\times\overset{p-2}{\dots}\times K/K_1$, which is abelian, so $K_1'=K_2$.
%
%
%
In particular, this and Proposition \ref{prop:rists} imply that  
$\rst_{\G}(n)'=\rst_{\G}(n+1)$ for each $n\geq 1$. 

\begin{corollary}
	For every prime $p\geq 5$, the GGS-group $\G\leq\Aut T$ with constant vector has the $p$-CSP, but not the CSP.
\end{corollary}

Let us now prove the remaining case, so that from now on $p=3$. The following result can be found in \cite{bartholdi-grigorchuk:parabolic}.

\begin{lemma}\label{K2=G''}
Let $\G$ and $K$ be as before, then we have
$$\psi(\G'')=K'\times K'\times K'.$$
\end{lemma}
\begin{proof}
One inclusion is clear because $\psi(\G')\leq K\times K\times K$ by \cite[Lemma 4.2 (iii)]{alcober-zugadi:hausdorff}.
For the other one, observe that $\psi([b,a])=(y_1,1,y_1^{-1})$ and $\psi([b^{-1},a]^a)=(y_0,y_0^{-1},1)$.
 Thus $\psi([[b,a],[b^{-1},a]^a])=([y_0,y_1],1,1)$ and, since $K'=\langle [y_0,y_1]\rangle ^{\G}$, the result follows.
\end{proof}

In order to apply Theorem \ref{thm:csp} with $R=K_1=K'$ and $H=K_2=\G''$ we must check that $K''\geq\psi^{-1}(\G''\times \G''\times\G'').$

\begin{proposition}\label{gamma3_K}We have $\G''\leq\gamma_3(K).$
\end{proposition}
\begin{proof}
Since $\G'=\langle[a,b]\rangle^{\G}$, we have 
$$\G''=\langle [[a,b],[a,b]^g]\mid g\in\G\rangle^{\G}.$$
Because $\gamma_3(K)$ is normal in $\G$, it suffices to prove that $ [[a,b],[a,b]^g]\in\gamma_3(K)$ for every $g\in\G$.
We already know that $\G/K\cong C_3$ and we can take as coset representatives $\{1,a,a^2\}$.
Write $g=ka^i$ with $i\in \F_3$ and $k\in K$. If $i=0$ there is nothing to prove, because 
\begin{align*}
[[a,b],[a,b]^g]&=[[a,b],[a,b][a,b,g]]\\
&=[[a,b],[a,b,g]],
\end{align*}
and since $G'\leq K$, the element belongs to $\gamma_3(K)$.

Suppose that $g=ka^i$ with $i=1,2$ and $k\in K$. Now we have

\begin{align*}
[[a,b],[a,b,ka^i]]&=[[a,b],[a,b,a^i][a,b,k]^{a^i}]\\
&=[[a,b],[a,b,k]^{a^i}][[a,b],[a,b,a^i]]^{[a,b,k]^{a^i}}.
\end{align*}

It is clear that the first factor is in $\gamma_3(K)$. On the other hand,
\begin{align*}
\psi([a,b])&=(b^{-1}a,1,a^{-1}b),\\
\psi([a,b,a])&=((a^{-1}b)^2,b^{-1}a,b^{-1}a),\\
\psi([a,b,a^2]&=(a^{-1}b,a^{-1}b,(b^{-1}a)^2),
\end{align*}
 imply that the second factor is trivial for $i=1,2$.
\end{proof}

\begin{proposition}
We have $\psi(K'')\geq \G''\times\G''\times \G''.$
\end{proposition}
\begin{proof}
By Proposition \ref{gamma3_K}, it suffices to prove the containment 
$$\psi(K'')\geq\gamma_3(K)\times\gamma_3(K)\times\gamma_3(K).$$

Since $\G$ is weakly regular branch over $K'$, we know that for every $k_1\in K'$ there is some $g_1\in K'$ such that $\psi(g_1)=(k_1,1,1)$. On the other hand, since $\psi([y_0,y_1])=(y_2,y_0,y_1)$  we get that $K'$ is subdirect in $K\times K\times K$. Thus, for every $k_2\in K$ there is some $g_2\in K'$ such that $\psi(g_2)=(k_2,*,*).$ Finally, we obtain $$\psi([g_1,g_2])=([k_1,k_2],1,1),$$
and the result follows.
\end{proof}

Now we can apply Theorem \ref{thm:csp} with $R=K'=K_1$ and $H=\G''=K_2$. 
By Proposition~\ref{prop:K1/K2_csp}, we only need to prove that $K_2=\G''$ has the $p$-CSP modulo $\psi^{-1}(\G''\times \G''\times \G'')$. 
Lemma~\ref{K2=G''} implies that  $\G'' /\psi^{-1}(\G''\times \G''\times \G'')\cong K_1/K_2\times K_1/K_2\times K_1/K_2$, and then using again Proposition \ref{prop:K1/K2_csp} the result follows.
\subsection{Example: Basilica group}
This group was defined by R. Grigorchuk and A. Zuk in \cite{grigorchuk-zuk:torsion-free}. 
In the same paper they prove that this group is torsion-free and weakly branch. 
We recall here the definition and some auxiliary results proved there. 
\begin{definition}
Let $T$ be the binary tree. The Basilica group $G$ is generated by two automorphisms $a$ and $b$ defined recursively as follows:
$$a=(1,b) \qquad
b=(1,a)\varepsilon$$ 
where $\varepsilon$ denotes the swap at the root.
\end{definition}
\begin{lemma}\label{basilica:multi-lemma}
Let $G$ be the Basilica group. Then,
\begin{enumerate}
\item $G$ acts transitively on all levels of $T$,
\item $\psi(G')\geq G'\times G'$, so $G$ is weakly branch over $G'$,
\item $G'=\psi^{-1}(G'\times G')\rtimes \langle [a,b]\rangle$,
\item $G/G'=\langle a\rangle\times\langle b\rangle\cong\Z\times\Z$,
\item $G$ is torsion-free.
\end{enumerate}
\end{lemma}
Since $G/G'\cong \Z\times\Z$, and all quotients by level stabilizers are 2-groups, $G$ does not have the congruence subgroup property.
We show below that it has the $2$-CSP.
\begin{lemma}
Let $A:=\langle a\rangle^G$ and $B:=\langle b\rangle^G$. Then $G'=A\cap B$.
\end{lemma}
\begin{proof}
 $G'=\langle [a,b]\rangle^G$ and $[a,b]\in A\cap B$, so $G'\leq A\cap B$.
 Since $G/(A\cap B) \cong G/A \times G/B \cong \Z \times \Z$, the result follows.
\end{proof}
\begin{lemma}\label{basilica:rist}
With $A$ and $B$ as above we have
\begin{enumerate}
\item $\rst_G(1)=A$ with $\psi(A)=B\times B$,
\item $\psi_{n-1}(\rst_G(n))=G'\times\overset{2^{n-1}}{\dots}\times G'$.
\end{enumerate}
\end{lemma}
\begin{proof}

The first item is Lemma 3 of \cite{grigorchuk-zuk:torsion-free}.
For the rest of the levels, the fact that $$\psi(\rst_G(n))=(\rst_G(n-1)\times\rst_G(n-1))\cap \psi(\rst_G(n-1))$$
for every $n$, implies in particular that
 $\psi(\rst_G(2))=(A\times A)\cap (B\times B)=G'\times G'$.
 The claim follows because $\psi(G')\geq G'\times G'$.
\end{proof}

It was pointed out to us by D. Francoeur, and separately by B. Klopsch  and H. Sasse, that Lemma 8 of \cite{grigorchuk-zuk:torsion-free} is inaccurate, which also affects the proof of Lemma 9 in that paper (although not the statement). 
Since we will make use of these results, we provide corrected versions. 

\begin{lemma}\label{lem:correct_lemmas89}
The following hold:
\begin{enumerate}
\item $\psi(G'')=\gamma_3(G)\times\gamma_3(G)$;
\item $\gamma_3(G)/G''\cong \Z^2$;
\item $G'/G''\cong \Z^3$.
\end{enumerate}
\end{lemma}
\begin{proof}
First note that $G' \ni [a,b^{-1}]=(b,b^{-1})$ and, since $([b,a],1), (1,[b,a]), \in G'$ by Lemma \ref{basilica:multi-lemma}, we obtain $([[b,a],b],1),   (1,[[b,a],b])\in G''$. 
It is easy to check that $[[b,a],a]=1$, so $\gamma_3(G)=\langle [[b,a],b]\rangle ^G $.
This and that $\st_G(1)\leq G\times G$ proves one inclusion of the first item. 
For the other inclusion, it suffices to show that $G'/\gamma_3(G)\times\gamma_3(G)$ is abelian. 
For this, we use that $G'=\langle [a,b^{-1}]\rangle^G$. 
Then $[a, b^{-1}]^{a^{\pm 1}}=[a,b^{-1}]$. 
A calculation yields 
\begin{align*}
[a,b^{-1}]^{b^{2n}} &= (b, b^{-1})^{(a^n,a^n)}\equiv (b[b, a]^n, b^{-1}[b,a]^{-n}) \mod \gamma_3(G)\times\gamma_3(G) \\
[a,b^{-1}]^{b^{2n+1}} &=(b^{-1}[b^{-1}, a], b)^{(a^n,a^n)}\equiv (b^{-1}[b, a]^{-n-1}, b[b,a]^{n}) \mod \gamma_3(G)\times\gamma_3(G)
\end{align*} 
for every $n\in\N$.
Since $a$ commutes with the above elements modulo $\gamma_3(G)\times \gamma_3(G)$, the group $G' / \gamma_3(G)\times \gamma_3(G)$ is generated by the images of $[a,b^{-1}]^m$ for $m\in \Z$, and these clearly commute with each other, which implies that $G'' \leq \gamma_3(G) \times \gamma_3(G)$. 

To show the second item, we examine the image of $\langle [[a,b^{-1}], b]\rangle^G$ modulo $\gamma_3(G)\times \gamma_3(G)$:
 \begin{align*}
 x &:=[[a,b^{-1}],b]=(b^{-1}b^{-a}, b^2)\equiv (b^{-2}[b,a]^{-1}, b^2)\\
x^{b^{2n}} &= x^{(a^n, a^n)} \equiv (b^{-2}[b,a]^{-2n-1}, b^2[b,a]^{2n})\\
x^{b^{2n+1}} &= (x^b)^{(a^n, a^n)} \equiv (b^{2}[b,a]^{2n+2}, b^{-2}[b,a]^{-2n-1}).
\end{align*}
Since $a$ commutes with $x^{b^m}$ for all $m\in\Z$, these conjugates of $x$ generate $\gamma_3(G)$. 
Writing $y:=x^{b^2}x^b\equiv ([b,a]^{-1}, [b,a])$,  it is easily checked that $x^{b^{2n}}\equiv xy^{2n}$ and $x^{b^{2n+1}}\equiv x^{-1}y^{-2n-1}$ modulo $\gamma_3(G)\times \gamma_3(G)$, for all $n\in \Z$. 

Now, $[a,b]^n=(b^{na}, b^{-n})\equiv (b^n[b,a]^n, b^{-n}) \mod \gamma_3(G)\times\gamma_3(G)$ for every $n\in\Z$.
Suppose that  $[a,b]^n\in\gamma_3(G)$ for some $n\in\Z$. 
Then there exist $r, s\in\Z$ such that 
$$x^ry^s\equiv (b^{-2r}[b,a]^{-r-s}, b^{2r}[b,a]^s) \equiv (b^n[b,a]^n, b^{-n}) \equiv [a,b]^n \mod \gamma_3(G)\times\gamma_3(G). $$
Comparing the second coordinate, and using that $b^m\in G'$ if and only if $m=0$ (see Lemma \ref{basilica:multi-lemma}), we get $2r=-n$ and $s=0$. 
Using this for the first coordinate yields that $b^{n+2r}\equiv [b,a]^{-r-n} \mod \gamma_3(G)$. 
But this means that $n=r=-2r$, so $r=0$ and $n=0$. 
This now easily implies that $x$ and $y$ are of infinite order modulo $\gamma_3(G)\times \gamma_3(G)$, proving the second item.

Put $c:=[a,b^{-1}]G''$,  $d:=[a,b^{-1}][a,b^{-1}]^bG''=([a,b], 1)G''$ and  $e:=[a,b^{-1}][a,b^{-1}]^{b^{-1}}G''=(1,[a,b])G''.$
Then $c, d, e$ clearly generate $G'/\left( \gamma_3(G)\times \gamma_3(G) \right) $ and the above argument shows that they are of infinite order and linearly independent, proving the third item. 
\end{proof}

In view of the fact that $\psi(\gamma_3(G))\geq\psi(G'')=\gamma_3(G)\times \gamma_3(G)$, we will take $R=G'$ and $H=\gamma_3(G)$  to apply Theorem \ref{thm:csp}. 
Note that we even have $L_n\leq \rst_G(n)'$ for all $n\in \N$, in the notation of that theorem.
It only remains to show that $G$ and $\gamma_3(G)$ have the $2$-CSP modulo $\gamma_3(G)$ and $\psi^{-1}(\gamma_3(G)\times\gamma_3(G))$, respectively.
The rest of this section is devoted to proving this.
\begin{proposition}\label{prop:infinite_mod_gamma3}
The quotient $G'/\gamma_3(G)$ is infinite cyclic and $G/\gamma_3(G)$ is isomorphic to the integral Heisenberg group.
\end{proposition}
\begin{proof}
The first statement follows from the last two items of Lemma \ref{lem:correct_lemmas89}. 
The second follows from the first statement and the fact that $G/G' \cong \Z^2$.
%
%
%
%
\end{proof}

\begin{lemma}\label{lem:g0g1inG'}
If $g\in G'$ is such that $\psi(g)=(g_1,g_2)$ then $g_1g_2\in G'$. Similarly if $g\in G'\st_G(n)$ then $g_1g_2\in G'\st_G(n-1)$.
\end{lemma}
\begin{proof}
Define $\varphi: G\times G\longrightarrow G/G'$ by $(g_1,g_2) \mapsto g_1g_2 G'$. 
We claim that $G'\leq \ker(\varphi\circ\psi)$.
Clearly, $\psi^{-1}(G'\times G')$ is contained in the kernel and, since $G'\geq\psi^{-1}(G'\times G')$, it suffices to check that the image is in the kernel for the generators of $G'$ modulo $\psi^{-1}(G'\times G')$, that is, for $[a,b]$. The result follows because $\psi([a,b])=(b^{a},b^{-1})$.
\end{proof}

\begin{proposition}\label{prop:GhasCSPmodgamma3}
The group $G$ has the $2$-CSP modulo $\gamma_3(G)$. 
\end{proposition}
\begin{proof}
	It suffices to prove that $G$, $A$, and $G'$ have the $2$-CSP modulo $A$, $G'$ and $\gamma_3(G)$, respectively, and apply Lemma \ref{lem:csp_transitivity} twice. 
	Since $G/A\cong A/G'\cong G'/\gamma_3(G)\cong \Z$, it is enough to show that $|G:A\st_G(n)|$, $|A:G'\st_{A}(n)|$ and $|G':\gamma_3(G)\st_{G'}(n)|$ tend to infinity with $n$. 
	Indeed, since in $\Z$ the subgroups of order a power of $2$ are totally ordered, this will imply that any normal subgroup $N$ of index a power of $2$ in, for instance, $A\leq N\leq G$ will satisfy $N\geq\st_G(n)A$ for some $n\in\N$. 
	
	 We first prove by induction that $b^{2^n}\notin A\st_G(2n+1)$. The base step, $b\notin A\st_G(1)=\st_G(1)$, is clear.
	  Now assume that $b^{2^{n-1}}\notin A\st_G(2n)$ and suppose for a contradiction that $b^{2^n}\in A\st_G(2n+1)$.
	  By Lemma \ref{basilica:multi-lemma}, we have 
	  $$A\st_G(2n+1)=\langle a\rangle G'\st_G(2n+1)=\langle a\rangle \langle [a,b]\rangle \psi^{-1}(G'\times G')\st_G(2n+1).$$
	   So there are $i, j\in \Z$ such that 
	    $[a,b]^ja^ib^{2^n} \in \psi^{-1}(G'\times G')\st_G(2n+1).$
	   Thus 
	   $$\psi([a,b]^ja^ib^{2^n})=((b^a)^ja^{2^{n-1}}, b^{i-j}a^{2^{n-1}})   \in G'\st_G(2n)\times G'\st_G(2n).$$ 
	   Consider $b^{i-j}a^{2^{n-1}}\in G'\st_G(2n)$.
	   As $\psi(b^{i-j}a^{2^{n-1}})=(a^{(i-j)/2}, a^{(i-j)/2}b^{2^{n-1}})$,
	    applying Lemma \ref{lem:g0g1inG'} yields $a^{i-j}b^{2^{n-1}}\in G'\st_G(2n-1)\leq A\st_G(2n-1)$. 
	   This implies that $b^{2^{n-1}}\in A\st_G(2n-1)$, a contradiction.
	   The claim follows by induction.
	   
	   This easily implies that $a^{2^n}\notin G'\st_G(2n+2)$ for each $n\in\N$.
	   Indeed,  $a^{2^n}=(1,b^{2^n})$ and, since $b^{2^n}\notin G'\st_G(2n+1)$, 
	   Lemma \ref{lem:g0g1inG'} yields that $a^{2^n}$ cannot be in $G'\st_G(2n+2)$.
	   
	   Finally, let us prove that $|G':\gamma_3(G)\st_{G'}(n)|$ tends to infinity with $n$. 
	   Suppose that it does not, so there exist $M, K\in\N$ such that $G'/\gamma_3(G)\st_{G'}(m)\cong \Z/2^K\Z$ for all $m\geq M$. 
	   In particular, $[a,b^{-1}]\equiv [a,b]^{-1}$ has order $2^K$ modulo $\gamma_3(G)\st_{G'}(m)$. 
	   By the proof of the second item of  Lemma \ref{lem:correct_lemmas89}, for each $m\geq M$ there exist $r_m, s_m\in \Z$ such that 
	   \begin{equation}\label{eq:[a,b]2Kingamma3st(m)}
	   [a,b^{-1}]^{2^K}\equiv (b^{2^K}, b^{-2^K})\equiv x^{r_m}y^{s_m}\equiv (b^{-2r_m}[b,a]^{-r_m-s_m}, b^{2r_m}[b,a]^{s_m})
	   \end{equation}	   
	   modulo $(\gamma_3(G)\times \gamma_3(G))\st_{G'}(m)\leq \gamma_3(G)\st_G(m-1)\times  \gamma_3(G)\st_G(m-1)$. 
	   In particular, $b^{2^K+2r_m}\in\G'\st_G(m-1)$.
	   We have seen above that $b^{2^n}\in G'\st_G(2n)\setminus G'\st_G(2n+1)$ for all $n\in\N$.
	   Thus $2^{\lfloor m/2 \rfloor}$ divides $2^K+2r_m$. 
	   In other words, denoting by $v(\cdot)$ the 2-adic valuation, $\lfloor m/2\rfloor \leq v(2^K+2r_m)$. 
	   If $v(2r_m)\neq K$ then $v(2^K+2r_m)= \min \{K, v(2r_m) \}\geq \lfloor m/2 \rfloor$, which is impossible for $m> 2K+1$ and so for those $m$ there must be some odd $t_m\in \Z$ such that $2r_m=2^Kt_m$. 
	   Thus equation \ref{eq:[a,b]2Kingamma3st(m)} implies that $b^{2^K+2^Kt_m}[b,a]^{2^K t_m + s_m}, b^{2^K+2^Kt_m}[b,a]^{s_m}\in \gamma_3(G)\st_G(m-1)$ for $m\geq \max \{M, 2K+2\}$, which in turn means that $[b,a]^{2^{K-1}t_m}\in\gamma_3(G)\st_G(m-1)$.
	   Since $G'/\gamma_3(G)\st_G(n)$ is a 2-group for all $n\in \N$, and $t_m$ is odd, we deduce that $[b,a]=[a,b]^{-1}$ has order at most $2^{K-1}$ modulo $\gamma_3(G)\st_G(m-1)$ for all $m\geq \max \{M,2K+2\}$, a contradiction. 
\end{proof}

\begin{proposition}
The group $\gamma_3(G)$ has the $2$-CSP modulo $\psi^{-1}(\gamma_3(G)\times \gamma_3(G))$.
\end{proposition}
\begin{proof}
This is proved like the previous result. 
In the proof of Lemma \ref{lem:correct_lemmas89}, we saw that $\psi(\gamma_3(G))$ is generated by $(b^{-2}[b,a]^{-1}, b^2)$  and $([b,a]^{-1}, [b,a])$
modulo $\gamma_3(G)\times \gamma_3(G)$, so it is also generated by $\alpha:=(b^{2}[b,a], b^{-2})$ and $\beta:=(b^{-2}, b^{2}[b,a]^{-1} )$. 

We first show that $\gamma_3(G)$ has the 2-CSP modulo $\langle \beta\rangle \gamma_3(G)\times \gamma_3(G)$.
Suppose for a contradiction that there exist $M, K\in\N$ such that $\gamma_3(G)/\langle \beta\rangle (\gamma_3(G)\times\gamma_3(G))\st_{\gamma_3(G)}(m) \cong \Z/2^K\Z$ for all $m\geq M$. 
That is, for all $m\geq M$ there exists $r_m\in \Z$ such that 
$$\alpha^{2^K}=(b^{2^{K+1}}[b,a]^{2^K}, b^{-2^{K+1}})\equiv \beta^{r_m}=(b^{-2r_m}, b^{2r_m}[b,a]^{-r_m})$$
modulo $(\gamma_3(G)\times\gamma_3(G))\st_{\gamma_3(G)}(m)\leq \gamma_3(G)\st_G(m-1)\times\gamma_3(G)\st_G(m-1)$.
In particular, 
$$b^{2^{K+1}+2r_m}[b,a]^{2^K}, b^{2^{K+1}+2r_m}[b,a]^{-r_m}\in \gamma_3(G)\st_G(m-1)$$
and therefore $[b,a]^{2^K}\gamma_3(G)\st_G(m-1) = [b,a]^{-r_m}\gamma_3(G)\st_G(m-1) $.
We saw in the proof of \ref{prop:GhasCSPmodgamma3} that $G'/\gamma_3(G)\st_G(n)\cong \Z/2^{t_n}\Z$ where $t_n$ tends to infinity with $n$, which means that
$-r_m=2^K$ for all large enough $m$. 
Thus $b^{2^{K+1}-2^{K+1}}[b,a]^{2^K} \in \gamma_3(G)\st_G(m-1)$ for all large enough $m$,%
contradicting that $G'$ has the 2-CSP modulo $\gamma_3(G)$. 

To show that $\langle \beta\rangle (\gamma_3(G)\times \gamma_3(G))$ has the 2-CSP modulo $\gamma_3(G)\times \gamma_3(G)$, 
suppose for a contradiction that there exist $M, K\in \N$ such that $\beta$ has order $2^K$ modulo $(\gamma_3(G)\times \gamma_3(G))\st(m)$
for all $m\geq M$. 
This means that $(b^{-2^{K+1}}, b^{2^{K+1}}[b,a]^{-2^K} )\in (\gamma_3(G)\times \gamma_3(G))\st(m)\leq (\gamma_3(G)\st(m-1)\times \gamma_3(G)\st(m-1))$ for all $m\geq M$. 
In particular, $b^{2^{K+1}}\in G'\st(m-1)$ for all $m\geq M$, a contradiction to the proof of \ref{prop:GhasCSPmodgamma3}. 

Lemma \ref{lem:csp_transitivity} now yields the result. 
%
%
%
%
%
%
%
%
%
\end{proof}

\bibliographystyle{plain}
\bibliography{almost_csp}

\end{document}